\pdfoutput=1
\documentclass[microtype]{gtpart}
\usepackage{microtype}
\usepackage{graphicx}
\usepackage[mathscr]{eucal}
\usepackage{amssymb}
\usepackage[dvipsnames]{xcolor}
\usepackage{enumerate, cite}

\newcommand{\br}{\mathbb{R}}

\newcommand{\bz}{\mathbb Z}
\newcommand{\bn}{\mathbb N}

\newcommand{\cc}{\mathcal C}

\newcommand{\vp}{\varphi}
\newcommand{\Si}{\Sigma}
\newcommand{\sig}{\sigma}

\newcommand{\E}{\mathscr E}

\newcommand{\ssm}{\smallsetminus}

\newcommand{\map}[1]{\mathrm{Map}(#1)}
\newcommand{\mappm}[1]{\mathrm{Map}^\pm(#1)}
\newcommand{\pmap}[1]{\mathrm{PMap}(#1)}
\newcommand{\clpmap}[1]{\overline{\mathrm{PMap}_c(#1)}}
\newcommand{\cpmap}[1]{{\mathrm{PMap}_c(#1)}}

\DeclareMathOperator{\Homeo}{Homeo}
\DeclareMathOperator{\genus}{genus}
\DeclareMathOperator{\Aut}{Aut}

\DeclareMathOperator{\Hom}{Hom}

\DeclareMathOperator{\pr}{pr}

\newtheorem{Thm}{Theorem}[section]
\newtheorem{Thm*}{Theorem}
\newtheorem{Prop}[Thm]{Proposition}
\newtheorem{Lem}[Thm]{Lemma}
\newtheorem{Cor}[Thm]{Corollary}
\newtheorem{Cor*}[Thm*]{Corollary}

\newtheorem{Prob}[Thm*]{Problem}

\theoremstyle{definition}
\newtheorem{Def}[Thm]{Definition}

\numberwithin{equation}{section}

\title{The first integral cohomology of  \\ pure mapping class groups}

\author{Javier Aramayona}
\address{Departmento de Matem\'aticas, Universidad Aut\'onoma de Madrid \& Instituto de Ciencias Matem\'aticas, CSIC \\ Madrid, Spain}
\email{javier.aramayona@uam.es}

\author{Priyam Patel}
\address{Department of Mathematics, University of Utah \\ Salt Lake City, UT 84112}
\email{patelp@math.utah.edu}

\author{Nicholas G. Vlamis}
\address{Department of Mathematics, CUNY Queens College \\ Flushing, NY 11367}
\email{nicholas.vlamis@qc.cuny.edu}

\begin{abstract}
It is a classical result that pure mapping class groups of connected, orientable surfaces of finite type and genus at least three are perfect.
In stark contrast, we construct nontrivial homomorphisms from infinite-genus mapping class groups to the integers.
Moreover, we compute the first integral cohomology group associated to the pure mapping class group of any connected orientable surface of genus at least 2 in terms of the surface's simplicial homology.
In order to do this, we show that pure mapping class groups of infinite-genus surfaces split as a semi-direct product.
\end{abstract}

\begin{document}

\maketitle


\vspace{-35pt}
\section{Introduction}

Throughout this article---unless otherwise noted---a surface refers to an orientable, connected, second-countable 2-manifold with (possibly empty) compact boundary. 
Associated to every such surface $S$ is its mapping class group, written $\map S$, which is the group of orientation-preserving homeomorphisms of $S$ up to isotopy, where all homeomorphisms and isotopies fix the boundary of $S$ pointwise. 
The {\em pure mapping class group}, written $\pmap S$, is the subgroup of $\map S$ whose elements induce the trivial permutation on the set of topological ends of $S$. 

The structure of $\map S$ is generally well-understood when $S$ is of finite type, that is, when $\pi_1(S)$ is finitely generated. 
In contrast, infinite-type mapping class groups have been less studied, although there has been a recent surge of activity in exploring their structure. 
The work to date has largely focused on highlighting the similarities between infinite-type mapping class groups and their finite-type counterparts \cite{AFP, BavardHyperbolic, BavardIsomorphisms, DurhamGraphs, PatelAlgebraic, HernandezAlexander}. Nevertheless, the comparison between the two have strong limitations \cite{BavardBig, PatelAlgebraic}.

The goal of this article is to exhibit another fundamental difference between the finite-type and the infinite-type settings. 
It is a classical result, that if $S$ is of finite type and has genus at least 3, then $\pmap S$ is perfect, that is, it has trivial abelianization (see \cite[Section 5.1]{FarbPrimer} for a proof).
It follows that every homomorphism to $\bz$ is trivial. 
Our main result asserts that in most infinite-genus cases, there exist nontrivial homomorphisms to $\mathbb{Z}$.
Let \( T \) denote a torus with an open disk removed, then a \emph{handle} in a surface refers to an embedded copy of \( T \). 
Intuitively, a topological end of a surface is \emph{accumulated by genus} if there is a sequence of handles escaping to the end.

\begin{Thm*}
\label{thm:rank}
Let \( S \) be a surface of genus at least two.
\begin{enumerate}
\item
\( S \) has at most one end accumulated by genus if and only if  \( H^1( \pmap S, \bz) \) is trivial.
\item
Let \( n \in \bn \)\footnote{For clarity, we set \( \bn = \{1, 2, 3, \ldots\} \); in particular, \( 0 \notin \bn \).}. 
The rank of \( H^1(\pmap S, \bz) \) is \( n  \) if and only if \( S \) has \( n+1 \) ends accumulated by genus. 
\item
\( S \) has infinitely many ends accumulated by genus if and only if the rank of \( H^1(\pmap S, \bz) \) is  infinite.
\end{enumerate}
\end{Thm*}

We note that the genus-2 restriction can be reduced to genus-1 by the subsequent work of Domat--Plummer \cite{DomatFirst}.
Moreover, the genus restriction is necessary as pure mapping class groups of (finite-type) planar surfaces surject onto nontrivial free groups; this is discussed further in \cite{DomatFirst} where explicit constructions, unique to infinite-type surfaces, are given.

In Theorem \ref{thm:rank} (and throughout the paper), \( \bz \) is treated as the trivial \( \pmap S \)-module, which allows us to identify \( H^1(\pmap S, \bz) \) with \( \Hom(\pmap S, \bz) \). 

The idea of constructing a nontrivial homomorphism \( \pmap S \to \bz \) is the following:  take a simple closed curve \( \gamma \) on \( S \) such that \( S \ssm \gamma \) is disconnected and such that there is an end accumulated by genus on either side of \( \gamma \) (note that \( \gamma \) is nontrivial in homology).
Now given an element \( f \in \pmap S \), we ``count" the net number of handles that \( f \) moves from the left of \( \gamma \) to the right of \( \gamma \); this will establish a homomorphism to \( \bz \). 
To formally perform this count, we introduce a \( \bz \)-coloring of the subgraph of the curve graph of \( S \) consisting of curves homologous to \( \gamma \); we then see that \( \pmap S \) acts on the colors by translations, yielding a homomorphism to \( \bz \) (this is the content of Section~ \ref{coloring}).
This construction mimics that of Gaster, Greene, and the third author in \cite{GasterColoring} of the coloring homomorphism \( \chi \) from the Torelli group \( \mathcal I (S_g) \) to \( H_1(S_g, \bz/(g-1)\bz) \), where \( S_g \) is the closed genus-\( g \) surface (the coloring homomorphism from \cite{GasterColoring} is itself a novel construction of the Chillingworth homomorphism).

The most basic example of a homeomorphism that has a non-zero net shift of handles from one side of a curve (such as \( \gamma \) above) to the other is the \emph{handle shift}, which was introduced in \cite{PatelAlgebraic} by the second and third authors (see Section \ref{handleshifts}).
In fact, we will see that the nontrivial image of a pure mapping class group in \( \bz \) is always generated by a handle shift.
For an example, let \( L \) be the ladder surface, the two-ended borderless surface with both ends accumulated by genus (see Figure \ref{fig:colors}), and let \( \tau \) be the translation homeomorphism satisyfing \( \langle \tau \rangle \backslash L \) is a torus, then (as we will see below) there is a unique homomorphism \( \vp \co \pmap L \to \bz \) such that \( \vp(\tau) = 1 \) (and if \( \psi\co \pmap L \to \bz \) is any other homomorphism, then \( \psi = n\vp \) for some \( n \in \bz \)). 
A handle shift is a natural generalization of this deck transformation.

Theorem \ref{thm:rank} is a corollary of Theorem \ref{thm:cohomology} below computing the first cohomology group of \( \pmap S \) with integral coefficients, which itself relies on a structural result (Theorem \ref{thm:decomposition} below) decomposing \( \pmap S \) as semi-direct product.
But, before getting to these more technical results, let us discuss some corollaries of Theorem \ref{thm:rank} comparing the finite- and infinite-type settings.

As documented in \cite[Problem 2.11(A)]{Kirby}, an open question posed by Ivanov asks whether mapping class groups of finite-type surfaces virtually surject onto \( \bz \). This question is directly related to another open question \cite[Problem 2.11(B)]{Kirby} asking whether mapping class groups have Kazhdan's property (T); a positive answer to the second implies a positive answer for the first.
By restricting the number of ends accumulated by genus, we see that such epimorphisms exist for a large class of mapping class groups of infinite-type surfaces:

\begin{Cor*}
If a surface \( S \) has finitely many ends with at least two accumulated by genus, then \( \map S \) virtually surjects onto \( \bz \). 
\end{Cor*}  

\begin{proof}
In this case, \( \pmap S \) is finite index and admits an epimorphism to \( \bz \) by Theorem~\ref{thm:rank}. 
\end{proof}

(We should note that property (T) is discussed in the context of locally compact topological groups, which mapping class groups of infinite-type surfaces are not (see \cite{AramayonaBig})).

For another comparison, in the setting of pure mapping class groups of finite-type surfaces, the first author and Souto \cite{AramayonaHomomorphisms} showed that there are many restrictions on the existence of homomorphisms between pure mapping class groups; in particular, assuming certain genus bounds, all such homomorphisms arise from embeddings of surfaces.
In stark contrast, we have:

\begin{Cor*}
Let \( S \) be a surface with at least two ends accumulated by genus.
If \( S' \) is any other surface such that \( \pmap {S'} \) is nontrivial, then there exists a nontrivial homomorphism \( \pmap S \to \pmap{S'} \). 
\end{Cor*}

\begin{proof}
To see this, simply take a nontrivial homomorphism \( \pmap S \to \bz \) and then any nontrivial homomorphism \( \bz \to \pmap{S'} \).
\end{proof}

Let us now turn to our main results from which the above discussion stems. 
In particular, we construct all homomorphisms of the form \( \pmap S \to \bz \).
More concretely, let \( \hat S \) be the surface obtained from \( S \) by ``filling in'' the planar ends and capping off the boundary components of \( S \) (this is made more precise in Section \ref{classification}).
We let \( H_1^{sep}(\hat S, \bz) \) denote the subgroup of \( H_1( \hat S, \bz) \) generated by homology classes that can be represented by simple separating closed curves on the surface.

Equip the group of self-homeomorphisms of \( S \) with the compact-open topology and give \( \map S \) the corresponding quotient topology.
Let \( \cpmap S \) be the subgroup of \( \pmap S \) consisting of compactly-supported mapping classes (see Section \ref{mcg}), or equivalently, the subgroup of \( \pmap S \) generated by Dehn twists.
Let \( \clpmap S \) be the closure of \( \cpmap S \) and note that \( \clpmap S \) is normal in \( \pmap S \), which allows us to define the group \( A_S = \pmap S / \clpmap S \). 
We will prove:

\begin{Thm*}
\label{thm:cohomology}
If \( S \) a surface of genus at least two, then there exists a natural isomorphism between \( H_1^{sep}(\hat S, \bz) \) and \( H^1(\pmap S, \bz) \).
Moreover, every homomorphism from \( \pmap S \) to \( \bz \) factors through \( A_S \).
\end{Thm*}

The homomorphism in Theorem \ref{thm:cohomology} is constructed and shown to be injective in Proposition~\ref{prop:definition}. Surjectivity is more subtle, and its proof requires a deeper investigation of the structure of PMap(S), together with an application of descriptive set theory (see Theorem \ref{thm:isomorphism}, and the discussion below Corollary \ref{cor:semidirect}).
Applying basic properties of homology, we obtain Theorem~\ref{thm:rank} directly from Theorem~\ref{thm:cohomology}.

A recent result of Bavard, Dowdall, and Rafi \cite{BavardIsomorphisms} proving a conjecture of the second and third author from \cite{PatelAlgebraic} shows that the isomorphism type of either \( \map S \)  or \( \pmap S \) determines the homeomorphism type of \( S \) and vice versa.
Theorem~\ref{thm:rank} serves as another example of the direct link between the algebra of \( \pmap S \) and the topology of \( S \).

In order to prove Theorem \ref{thm:cohomology}, we first prove a structural result about \( \pmap S \).
In \cite[Theorem 4]{PatelAlgebraic}, the second and third author showed that \( \clpmap S = \pmap S \) if and only if \( S \) has at most one end accumulated by genus (recall that \( \cpmap S \) is the group generated by Dehn twists).
Moreover, when \( S \) has more than one end accumulated by genus, it was shown that the set of Dehn twists together with the set of handle shifts (see Section \ref{handleshifts}) generate a dense subgroup of \( \pmap S \).
The next result is a precise measure of the role of handle shifts in the failure of Dehn twists to topologically generate \( \pmap S \).

Let \( H^1_{sep}(\hat S, \bz) \) be the subgroup of \( H^1(\hat S, \bz) \) that is identified with \( \Hom(H_1^{sep}(\hat S, \bz), \bz) \) under the natural isomorphism \( H^1(\hat S, \bz) \to \Hom(H_1(\hat S, \bz), \bz) \) given by the universal coefficient theorem.

\begin{Thm*}
\label{thm:decomposition}
Let \( S \) be a surface with at least two ends accumulated by genus.
There is an injection \( \kappa\co H^1_{sep}(\hat S, \bz) \to \pmap S \) such that \( \pi \circ \kappa \co H^1_{sep}(\hat S, \bz) \to A_S \) is an isomorphism, where \( \pi \co \pmap S \to A_S \) is the natural projection.
Moreover, there exist pairwise-commuting handle shifts \( \{h_i\}_{i=1}^r \) such that
\[
\kappa(H^1_{sep}( \hat S, \bz)) = \prod_{i=1}^r \langle h_i \rangle,
\]
where \( r \in \bn \cup \{\infty\} \) denotes the rank of \( H_1^{sep}( \hat S, \bz) \). 
\end{Thm*}

Note that in Theorem \ref{thm:decomposition} the rank of \( H_1^{sep}(\hat S, \bz) \) can be (countably) infinite, so it is necessary to use the direct product.
As an immediate corollary of Theorem \ref{thm:decomposition}, we obtain:

\begin{Cor*}
\label{cor:semidirect}
For any surface \( S \),
\[ \pmap S = \clpmap S \rtimes \kappa(H^1_{sep}(\hat S, \bz)). \]
\end{Cor*}

Note that:
\begin{itemize}
\item
 if \( S \) has at most one end accumulated by genus, then Corollary \ref{cor:semidirect} is a restatement of \cite[Theorem 4]{PatelAlgebraic}.
 \item
 if \( S \) has \( n \) ends accumulated by genus with \( 2\leq n < \infty \), then \( \kappa(H^1_{sep}(\hat S, \bz)) \cong \bz^{n-1} \).
 \item
 if \( S \) has infinitely many ends accumulated by genus, then \( \kappa(H^1_{sep}(\hat S, \bz)) \cong \bz^\bn \), where \( \bz^\bn \) is the Baer-Specker group.
\end{itemize}

The proof of surjectivity in Theorem \ref{thm:cohomology} follows from Corollary \ref{cor:semidirect}, together with a classical result from descriptive set theory. 
It follows from recent work of Hernandez--Morales--Valdez \cite{HernandezIsomorphisms} and Bavard--Dowdall--Rafi \cite{BavardIsomorphisms} on the automorphism group of the curve graph that \( \pmap S \) is Polish, that is, separable and completely metrizable (see Section \ref{polish} for more detail). 
This makes available a result of Dudley \cite[Theorem 1]{DudleyContinuity}, which in our context implies that every homomorphism \( \pmap S \to \bz \) is continuous. 
The subgroup \( \cpmap S \) can be realized as a direct limit of finite-type mapping class groups, each of which is perfect (when the genus is at least 3) and hence is itself perfect. 
Therefore elements of \( H^1(\pmap S,\bz ) \) must vanish on \( \cpmap S \), and hence by Dudley must also vanish on \( \clpmap S \). 
Surjectivity then amounts to specifying the image of a homomorphism restricted to the abelian subgroup \( A_S \cong H^1_{sep}(S, \bz) \).

This is, to the authors' knowledge, the first application of descriptive set theory to the study of mapping class groups.
We note that descriptive set theory has been previously used in the study of homomorphisms of homeomorphism groups of compact manifolds, for instance see the work of Rosendal--Solecki \cite{RosendalAutomatic}, Rosendal \cite{RosendalAutomatic2}, and Mann \cite{MannAutomatic}.

Let us consider the abelianization of \( \pmap S \).
As a direct consequence of Corollary \ref{cor:semidirect}, we have

\begin{Cor*}
\label{cor:abelian}
For any surface \( S \), the group \( A_S = \pmap S / \clpmap S \) is abelian and hence
\[
[\pmap S, \pmap S] \leq \clpmap S,
\]
where \( [\pmap S, \pmap S] \) denotes the commutator subgroup of \( \pmap S \). 
\end{Cor*}

In the first version of this article, it was conjectured that \( \clpmap S \) is perfect when the genus of \( S \) is at least 3 and the abelianization of \( \pmap S \) is \( A_S \); however, Domat \cite{DomatBig} has disproven this conjecture and so we give the following problem:

\begin{Prob}
Let \( S \) be an infinite-type surface.
Compute the abelianization of \( \clpmap S \) and \( \pmap S \).
\end{Prob}

\subsection*{Acknowledgements}
\vspace{-12pt}
The authors thank Kasra Rafi and Justin Lanier for comments on an earlier draft.
The third author thanks Chris Leininger for a helpful and influential conversation predating the project.
The authors thank the referees for their valuable comments improving the exposition of this paper.

This project began during a visit of the third author to the first author's institution that was funded by the GEAR Network and as such:
the third author acknowledge(s) support from U.S. National Science Foundation grants DMS 1107452, 1107263, 1107367 "RNMS: GEometric structures And Representation varieties" (the GEAR Network).
The first author was supported by grants RYC-2013-13008 and MTM2015-67781.
The second author was supported by NSF DMS 1812014 \& 1937969.
The third author was supported in part by NSF RTG grant 1045119. 

\section{Preliminaries}
\label{background}

All surfaces appearing in the statements of results are assumed to be Hausdorff, connected, orientable, second countable, and to have (possibly empty) compact boundary; however, in proofs, there are several occasions where we will allow surfaces to have non-compact boundary.
The notion of a simple closed curve is critical to the study of mapping class groups, so for clarity, a simple closed curve on a surface \( S \) is an embedding of the circle \( \mathbb{S}^1 \) in \( S \).
A simple closed curve is \emph{trivial} if it is homotopic to a point;  it is \emph{peripheral} if it is either homotopic to a boundary component or bounds a once-punctured disk; it is \emph{separating} if its complement is disconnected. 

We will routinely abuse notation and conflate a mapping class with a representative homeomorphism and a simple closed curve with its isotopy class.

\subsection{Topological ends and the classification of surfaces}
\label{classification}

The classification of surfaces of infinite type relies on the notion of a topological end (first introduced by Freudenthal).
For a reference on the following material see \cite{RichardsClassification}.

\begin{Def}
The \emph{space of ends} of a surface \( S \), denoted \( \E(S) \), is the set
\[
\E(S) = \projlim_K(\pi_0(S\ssm K)),
\]
where the projective limit is taken over the set of all compact subsets of \( S \), which is directed with respect to inclusion.
Further, by equipping \( \pi_0( S\ssm K) \) with the discrete topology for each compact set \( K \), the set \( \E(S) \) can be given the limit topology, that is, the coarsest topology such that the projection maps \( \vp_K \co \E(S) \to \pi_0(S\ssm K) \) given by the universal property of projective limits  are continuous.

The \emph{Freudenthal compactification} of  \( S \), denoted \( \bar S \), is the set \( \bar S = S \sqcup \E(S) \) equipped with the topology generated by the basis consisting sets of the form \( U \cup U^* \), where \( U \) is an open set of \( S \) with compact boundary and \( U^* = \vp_{\partial U}^{-1}([U]) \). 
\end{Def}

An end \( e \) of \( S \) is \emph{accumulated by genus} if every neighborhood of \( e \) in \( \bar S \) has infinite genus.
The set of all ends accumulated by genus is a closed subset of \( \E(S) \) and is denoted \( \E_\infty(S) \).
Each end  \( e \in \E(S) \ssm \E_\infty(S) \) is \emph{planar}, that is, \( e \) has a neighborhood in \( \bar S \) homeomorphic to a disk.
The surface
\[
\hat S = \bar S \ssm \E_\infty(S)
\] 
is a surface in which every end is accumulated by genus and can be thought of as being obtained from \( S \) by ``forgetting'' or ``filling in'' the planar ends.
(If \( S \) has finite genus, then \( \bar S = \hat S \).)

The classification of surfaces states that a surface \( S \) is determined up to homeomorphism by the quadruple \( (g,b,\E(S), \E_\infty(S)) \), where \( g \) is the genus of \( S \) and \( b \) is the number of boundary components of \( S \).

\subsection{Mapping class groups}
\label{mcg}

Let \( \Homeo_\partial (S) \) denote the group of self-homeomorphisms of \( S \) that fix the boundary of \( S \) pointwise.
The \emph{extended mapping class group} of \( S \), denoted \( \mappm S \), is defined to be 
\[
\mappm S = \Homeo_\partial( S ) / \sim,
\]
where \( f \sim g \) if they are isotopic through homeomorphisms fixing \( \partial S \) pointwise.
If \( \partial S \) is empty, then the \emph{mapping class group}, denoted \( \map S \), is the index two subgroup of orientation-preserving mapping classes; otherwise, every element of \( \Homeo_\partial (S) \) is orientation-preserving and we identify \( \map S \) with \( \mappm S \). 
The \emph{pure mapping class group} is the kernel of the action of \( \map S \) on \( \E(S) \).

Given \( f \in \Homeo_\partial(S) \), we say \( f \) is \emph{compactly supported} if the closure of the set \( \{ x \in S \co f(x) \neq x \} \), called the \emph{support of \( f \)}, has compact closure, or equivalently, if \( f \) is the identity outside of a compact set.
We say a mapping class is \emph{compactly supported} if it has a compactly-supported representative.

The \emph{compactly-supported pure mapping class group} of \( S \), denoted \( \cpmap S \), is the subgroup of \( \map S \) consisting of compactly-supported mapping classes.

Equipped with the compact-open topology, \( \Homeo_\partial (S) \) becomes a topological group.
The \emph{compact-open topology} on \( \mappm S \) is defined to be the quotient topology induced by the epimorphism
\[
\mathrm{Homeo}_\partial(S) \to \mappm S.
\]
We will view any subgroup of \( \mappm S \) as a topological group by equipping it with the subspace topology, which will also be referred to as the compact-open topology.

\subsection{Handle shifts}
\label{handleshifts}

The abelian group in the decomposition of \( \pmap S \) presented in Corollary \ref{cor:semidirect} is generated by handle shifts, which were introduced in \cite{PatelAlgebraic}.  
We recall their definition now:
Let \( \Sigma \) be the surface obtained by gluing handles onto \( \br \times [-1,1] \) periodically with respect to the transformation \( (x,y) \mapsto (x+1,y) \).
Let \( \sig \co \Sigma \to \Sigma \) be the homeomorphism up to isotopy determined by requiring
\begin{enumerate}
\item \( \sig(x,y) = (x+1, y) \) for \( (x,y) \in \br \times [-1+\epsilon, 1-\epsilon] \) for some \( \epsilon > 0 \) and
\item \( \sig(x,y) = (x,y) \) for \( (x,y) \in \br \times  \{-1,1\} \). 
\end{enumerate}

We say a homeomorphism \( h\co S \to S \)  is a \emph{handle shift} if there exists a proper embedding \( \iota\co \Sigma \to S \) such that 
\[
h(x) = \left\{
\begin{array}{ll}
(\iota \circ \sig \circ \iota^{-1})(x) & \text{if } x \in \iota(\Sigma) \\
x & \text{otherwise}
\end{array}\right.
\]

We will say that $h$ is {\em supported} on $\iota(\Sigma)$, or simply on $\Sigma$, to relax notation. 
Since the embedding \( \iota \) is required to be proper, it induces a map \( \iota_\infty \co \E(\Si) \to \E_\infty(S) \).
We therefore see that a handle shift \( h \) has an attracting and repelling end, which we denote \( h^+ \) and \( h^- \), respectively.
That is, for any \( x \) in the interior of \( \iota(\Sigma) \) we have
\[
\lim_{n\to\infty}  h^n(x) = h^+ \text{ and } \lim_{n\to-\infty}  h^n (x) = h^{-},
\]
where the limit is taken in the Freudenthal compactification \( \bar S \) of \( S \).
Note that if \( h_1 \) and \( h_2 \) are homotopic handle shifts, then \( h_1^+ = h_2^+ \) and \( h_1^- = h_2^- \); therefore, there are well-defined notions of attracting and repelling ends for the mapping class associated to a handle shift. 
Note that it is possible for \( h^+ = h^- \); in this case, as a corollary to Theorem \ref{thm:cohomology}, \( h \) is a limit of compactly supported mapping classes.

\begin{figure}[t]
\center
\includegraphics{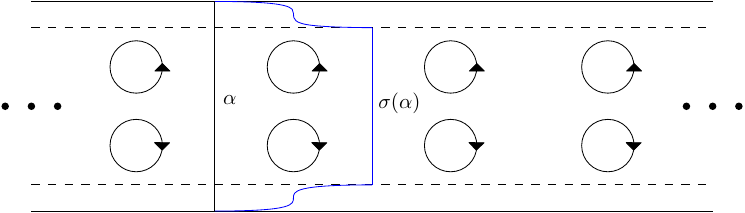}
\caption{The circles are identified vertically to obtain \( \Sigma \).}
\label{fig:handleshift}
\end{figure}

\subsection{Polish groups}
\label{polish}

A topological group is \emph{Polish} if its underlying topological space is separable and completely metrizable.
(See \cite{KechrisClassical} for a reference to the theory of Polish spaces/groups and for the following material.)
Let \( \Gamma \) be a graph with a countable vertex set.
Given any subset \( A \) of \( \Gamma \), define
\[
U(A) = \{ g \in \Aut(\Gamma) \co g(a) = a \text{ for all } a \in A \}.
\]
We can then equip \( \Aut(\Gamma) \) with the \emph{permutation topology}, which is generated by all \( \Aut(\Gamma) \)-translates of the pointwise stabilizers \( U(A) \), where \( A \) is a finite subset of \( \Gamma \).
This makes \( \Aut(\Gamma) \) into a topological group; furthermore, it is clear that \( \Aut(\Gamma) \) is separable (in fact, it is second countable).

\begin{Lem}
\label{lem:polish}
\( \Aut(\Gamma) \) with the permutation topology is Polish.
\end{Lem}

\begin{proof}
As \( \Gamma \) is countable, choose an enumeration \( \{x_i\}_{i\in \bn} \) of \( \Gamma \).
Let \( d \co \Aut(\Gamma) \times \Aut(\Gamma) \to \br \) be given by
\[
d(g,g') = \inf_{n\in \bn\cup\{0\}} \{2^{-n} \co g(x_i) = g'(x_i) \text{ for all } i < n \}.
\]
It turns out that \( d \) is not complete; however, we leave as an exercise to check that \( \rho(g,g') = d(g,g') + d(g^{-1},(g')^{-1}) \) is a complete metric and the corresponding metric topology is  the permutation topology.
\end{proof}

Let us turn back to mapping class groups.
Let \( \cc(S) \) denote the curve graph of \( S \), that is, the graph whose vertices consist of nontrivial non-peripheral isotopy classes of simple closed curves on \( S \) and where two vertices are adjacent if they have disjoint representatives.
The following recent theorem is a generalization of a result of Ivanov \cite{IvanovAutomorphisms} (see also \cite{LuoAutomorphisms}) in the finite-type setting:

\begin{Thm}[Hernandez--Morales--Valdez \cite{HernandezIsomorphisms}, Bavard--Dowdall--Rafi \cite{BavardIsomorphisms}]
\label{thm:curves}
If \( S \) is of infinite type and without boundary, then \( \Aut(\cc(S)) = \mappm S \).
\end{Thm}

When \( S \) is borderless, the identification of \( \mappm S \) with \( \Aut(\cc(S)) \) gives a permutation topology on \( \mappm S \).
A straightforward application of the Alexander Method \cite[Proposition 2.8]{FarbPrimer} yields:

\begin{Prop}
\label{prop:equivalence}
Let \( S \) be a borderless surface.
The compact-open topology and  the permutation topology are the same topologies on \( \mappm S \).
\end{Prop}

As every closed subspace of a Polish space is Polish we obtain:

\begin{Cor}
\label{cor:polish}
Every closed subgroup of \( \mappm S \) is Polish.
In particular, \( \map S \), \( \pmap S \), and \( \clpmap S \) are Polish.
\end{Cor}

\begin{proof}
If \( S \) is borderless, then the result is an immediate consequence of Lemma \ref{lem:polish}, Theorem \ref{thm:curves}, and Proposition \ref{prop:equivalence}.
If \( S \) has boundary, then it is straightforward to embed \( S \) in a larger borderless surface \( S' \) so that \( \map S \) is a closed subgroup of \( \map{S'} \). 
\end{proof}


\section{Colorings and homomorphisms}
\label{coloring}

In this section, we will use colorings of subgraphs of the curve graph to build nontrivial homomorphisms of pure mapping class groups of infinite-genus surfaces to \( \bz \).
As noted in the introduction, the graph colorings given here are analogs of  the colorings presented in \cite[Section 5]{GasterColoring} and, similarly, the homomorphisms constructed in this section mimic the coloring homomorphisms of \cite[Section 6]{GasterColoring}.

Let \( S \) be any surface.
For a homology class  \(v \in H_1(S, \bz) \), we will regularly attribute properties of a representative curve to \( v \); for instance, \( v \) is \emph{simple} if it there exists a simple closed curve on \( S \) representing \( v \).
An element of \( H_1(S, \bz) \) is \emph{primitive} if it cannot be written as a positive multiple of another element.
If a surface is compact with at most one boundary component or non-compact, borderless, and with at most one end, then every primitive homology class is simple (see \cite[Proposition 6.2]{FarbPrimer}); however, in all other cases there are primitive homology classes that fail to be simple (indeed, if \( v \) and \( w \) are two linearly independent peripheral homology classes, then \( v+2w \) is primitive, but not simple). 

Let \( \gamma \) be an oriented separating simple closed curve on \( S \).
The homology class of \( \gamma \) is determined, up to orientation, by its partition of \( \E(S) \cup \partial S\); in particular,  each simple element \( v \in H_1^{sep}(S, \bz) \) determines a partition of \( \E(S) \) into two disjoint sets \( v^- \) and \( v^+ \), where \( v^- \) consists of the ends to the left of \( v \) and \( v^+ \) those to the right.
Note that if \( v \) is simple, then \( v = 0 \) if and only if one of \( v^- \) or \( v^+ \) is empty.

For a simple nonzero element \( v \in H_1(S, \bz) \), let \( \cc_v(S) \) denote the subgraph of \( \cc(S) \) induced on the set of vertices that can be oriented to represent \( v \).
A \emph{proper coloring} of a graph is a labelling of the vertices such that no two adjacent vertices share a label.

Now, let \( S \) be an infinite-genus surface. 
Fix a simple nonzero homology class \( v \in H_1^{sep}(S, \bz) \) and let \( \gamma \) be an oriented simple closed curve on \( S \) representing \( v \). 
Choose an end \( e_v \in v^- \).
We will now give a proper \( \bz \)-coloring \( \phi_\gamma \co \cc_v(S) \to \bz \).

Let \( c \in \cc_v(S) \).
As \( c \cup \gamma \) is compact, there exists a connected compact surface \( R \) such that each component of \( \partial R \ssm \partial S \) is separating and \( c \cup \gamma \subset \mathring R \), where \( \mathring R \) denotes the interior of \( R \).
Let \( \partial_0 \) be the boundary component of \( R \) such that \( e_v \) and \( \mathring R \) are on opposite sides of \( \partial_0 \). 
Let \( \mathfrak g_R(c) \) and \( \mathfrak g_R(\gamma) \) be the genera of the component of \( R\ssm c \) and \( R \ssm \gamma \) containing \( \partial_0 \).
Define
\[
\phi_\gamma(c) = \mathfrak g_R(c) - \mathfrak g_R(\gamma).
\]
It is straightforward to check that the definition of \( \phi_\gamma \) is independent of the choices of \( R \) and \( e_v \in v^- \).

\begin{figure}
\centering
\includegraphics{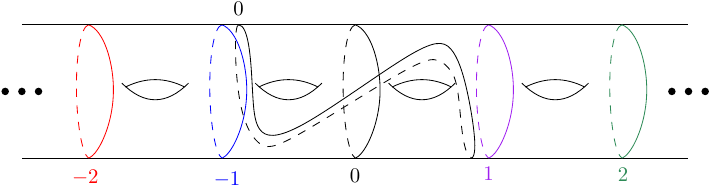}
\caption{On the two-ended infinite-genus surface with no planar ends \( L \), there is a unique nontrivial simple separating homology class \( v \).  Shown here is a---portion of a---\( \bz \)-coloring of \( \mathcal C_v(L) \).  Two of the curves shown are (necessarily) colored by 0. }
\label{fig:colors}
\end{figure}

\begin{Lem}
\( \phi_\gamma \) is a proper \( \bz \)-coloring of \( \cc_v(S) \).
\end{Lem}

\begin{proof}
If \( b,c \in \cc_v(S) \) are adjacent, then they cobound a compact surface \( F \) of positive genus satisfying 
\[
0 < \genus(F) = |\phi_\gamma(c) - \phi_\gamma(d)|.
\]
It follows that \( \phi_\gamma \) is a proper coloring.
\end{proof}

Since a nonzero, non-peripheral, simple, separating homology class \( v \) is determined, up to orientation, by its partition of  \( \E(S) \cup \partial S \), the action of \( \pmap S \) on \( \cc(S) \) restricts to an action on \( \cc_v(S) \) for each \( v \in H_1^{sep}(S, \bz) \).

\begin{Lem}
\label{lem:equations}
If \( v \in H_1^{sep}(S, \bz) \) is simple and nonzero, \( \gamma \) and \( \delta \) are oriented simple closed curves representing \( v \), and \( f \in \pmap S \), then the following equalities hold:
\begin{equation}
\label{eq:translation}
(\phi_\gamma\circ f) (c)- \phi_\gamma(c) = \phi_\gamma(f(\gamma)) \quad \text{for all } c\in \cc_v(S).
\end{equation}
\begin{equation}
\label{eq:constant}
\phi_\gamma(c)-\phi_\delta(c)  = \phi_\gamma(\delta) = -\phi_\delta(\gamma) \quad \text{for all } c\in \cc_v(S).
\end{equation}
\begin{equation}
\label{eq:subscript}
\phi_\gamma \circ f - \phi_\gamma = \phi_\gamma - \phi_{f(\gamma)}.
\end{equation}
\end{Lem}

\begin{proof}
Let \( c \in \cc_v(S) \).
Choose a compact surface \( R \) containing \( c, f(c), \gamma, f(\gamma) \) and such that each component of \( \partial R \ssm \partial S \) is separating. 
To see \eqref{eq:translation}, a direct computation yields:
\begin{align*}
\phi_\gamma(f(c)) 	&= \mathfrak g_R(f(c)) - \mathfrak g_R(\gamma) \\
				&= \mathfrak g_R(f(c))-\mathfrak g_R(f(\gamma))+ \mathfrak g_R(f(\gamma)) - \mathfrak g_R(\gamma)\\
				&= \phi_\gamma(c) + \mathfrak g_R(f(\gamma)) - \mathfrak g_R(\gamma)\\
				&= \phi_\gamma(c) + \phi_\gamma(f(\gamma))
\end{align*}

\eqref{eq:constant} follows from a near identical argument.
Combining \eqref{eq:translation} and \eqref{eq:constant} with \( \delta = f(\gamma) \) yields \eqref{eq:subscript}.
\end{proof}

Given a handle shift \( h \) and a simple element \( v \in H_1^{sep}(S, \bz) \), we say \( v \) \emph{cuts} \( h \) if \( h^+ \) and \( h^- \) are on opposite sides of \( v \). 
In light of Lemma \ref{lem:equations}, we are now in a position to define a homomorphism \( \vp_v \co \pmap S \to \bz \) for \( v \in H_1^{sep}(S, \bz) \) simple.

\begin{Prop}[Definition of \( \vp_v \)]
\label{prop:homomorphism}
Let \( S \) be an infinite-genus surface.
Let  \( v \in H_1^{sep}(S, \bz) \) be simple and nonzero.
If  \( \gamma \) is any oriented simple closed curve representing \( v \), then \[ \vp_v \co \pmap S \to \bz \] given by 
\[
\vp_v(f) = \phi_\gamma(f(\gamma)),
\]
is a well-defined homomorphism. 
In addition, \( \vp_v \) satisfies:
\begin{enumerate}

\item
\( \vp_v\left( \clpmap S\right) = 0 \).

\item \( \vp_{-v} = - \vp_v \).

\item 
If \( h \in \pmap S \) is a handle shift, then \( \vp_v(h) \neq 0 \) if and only if \( v \) cuts \( h \).

\item
\( \vp_v \neq 0 \). 

\end{enumerate}
\end{Prop}

\begin{proof}
Fix \( v \in H_1^{sep}(S, \bz) \). 
We first need to show that \( \phi_v \) is well-defined.
Let \( \gamma \) and \( \delta \) be oriented simple closed curves homologous to \( v \).
Using \eqref{eq:constant}, we have
\[
\phi_\gamma(f(\gamma)) 	= \phi_\delta(f(\gamma))-\phi_\delta(\gamma) = \phi_\delta(f(\delta))
\]

Hence, \( \vp_v \) is well defined.

If \( f, g \in \pmap S \), then
\begin{align*}
\vp_v(fg)	&= \phi_\gamma(fg(\gamma)) \text{  (by definition)} \\
		&= -\phi_{fg(\gamma)}(\gamma) \text{  (by \eqref{eq:constant})} \\
		&=  - \phi_{g(\gamma)}(\gamma) + \phi_{g(\gamma)}(\gamma) - \phi_{fg(\gamma)}(\gamma) \\
		&= - \phi_{g(\gamma)}(\gamma) + \phi_{g(\gamma)}\circ f(\gamma) - \phi_{g(\gamma)}(\gamma) \text{  (by \eqref{eq:subscript})} \\
		&= - \phi_{g(\gamma)}(\gamma) + \phi_{g(\gamma)}(f(g(\gamma))) \text{  (by \eqref{eq:translation}} \\
		&= \vp_v(g) + \vp_v(f) \text{  (by definition and well-definedness)}.
\end{align*}
This establishes \( \vp_v \co \pmap S \to \bz \) as a homomorphism.

To see (1), let \( \bar f \in \clpmap S \), let \( c \in \cc_v(S) \), and let \( f \in \cpmap S \) such that \( f(c) = \bar f(c) \).
Choose a compact surface \( R \) containing \( c \) and the support of \( f \) such that each component of \( \partial R \ssm \partial S \) is separating.
It must be that \( \mathfrak g_R(f(c)) = \mathfrak g_R(c) \).
In particular, if we orient \( c \) to represent \( v \), then
\[ 
0 = \phi_c(f(c)) = \phi_c(\bar f(c)) = \vp_v(\bar f). 
\]

(2) can be seen from the definitions but is also a direct consequence of \eqref{eq:constant} and  \eqref{eq:subscript}.

To see (3), first let \( h \in \pmap S \) be a handle shift cut by \( v \).
Let \( h \) be supported in the subsurface \( \Sigma \subset S \). 
We can then choose an oriented representative \( \gamma \) of \( v \) so that \(| \gamma \cap \partial \Sigma | = 2 \). 
In this case \( h(\gamma) \) is disjoint from \( \gamma \) and hence they cobound a positive genus subsurface yielding \( \vp_v(h) \neq 0 \). 
Conversely, if \( v \) does not cut a handle shift \( h \), then there is a representative \( \gamma \) of \( v \) disjoint from the support of \( h \).
It follows that \( \vp_v(h) = 0 \).

Now to (4):
If \( v \in H_1^{sep}(S, \bz) \) is nonzero, it partitions \( \E_\infty(S) \) nontrivially, and there exists a handle shift \( h \in \pmap S \) cut by \( v \); in particular, by (3), \( \vp_v(h) \neq 0 \).
\end{proof}

\textbf{Remark.}
If \( S \) has two ends accumulated by genus and \( v \in H_1^{sep}(S, \bz) \) is simple and nonzero, then the function \( d \co \cc_v(S) \to \br \) defined by \( d(a,b) = |\vp_v(a)-\vp_v(b)| \) is a metric.
Moreover, \( \cc_v(S) \) with this metric agrees with the metric on the graph \( \mathcal G (S) \) defined in \cite[Section 9]{DurhamGraphs}.
In \cite{DurhamGraphs}, this metric space and the corresponding action of \( \map S \) is used to show that \( \map S \) admits a left-invariant infinite-diameter pseudo-metric.

\section{Proofs of Theorems \ref{thm:cohomology} and \ref{thm:decomposition}}

In Section \ref{coloring}, we constructed an element of \( H^1(\pmap S, \bz) \) for each nonzero simple element of \( H_1^{sep}(S, \bz) \).
The remainder of the article is dedicated to showing that these homomorphisms span all of \( H^1(\pmap S, \bz) \). 
To this end, the goal of the next sequence of lemmas is to extend the definition of \( \vp_v \) to an arbitrary element \( v \in H_1^{sep}(S, \bz) \) by linearity.

\begin{Lem}
\label{lem:additive}
Let \( v_1, \ldots, v_n \) be simple homology classes in \( H_1^{sep}(S, \bz) \) such that there exist pairwise-disjoint oriented simple closed curves \( \gamma_1, \ldots, \gamma_n \) so that \( \gamma_i \) represents \( v_i \).
If \( v = \sum v_i \) is simple, then
\[
\vp_v = \sum_{i=1}^n \vp_{v_i}.
\] 
\end{Lem}

\begin{proof}

\begin{figure}[t]
\centering
\includegraphics{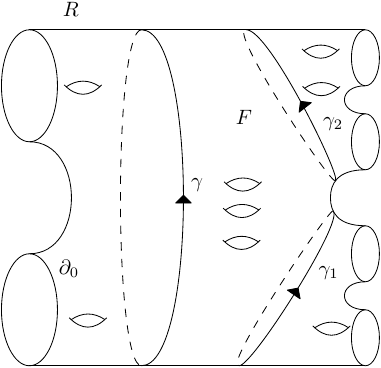}
\caption{\( R, \Sigma, \gamma, \gamma_1, \gamma_2, \partial_0 \) as in the proof of Lemma \ref{lem:additive}.}
\label{fig:additive}
\end{figure}

Let \( \gamma \) be an oriented simple closed curve representing \( v \).
By the hypothesis of the lemma, \( \gamma \) can be chosen to be disjoint from \( \gamma_i \) for all \( i \in \{1, \ldots, n \} \); furthermore, \( \gamma_1, \ldots, \gamma_n, \) and \( -\gamma \) cobound a compact surface \( F \).
By part (2) of Proposition \ref{prop:homomorphism}, we can assume without loss of generality that \( F \) is to the right of \( \gamma \) as in Figure \ref{fig:additive}. 

Now fix an element \( f \in \pmap S \) and choose a compact surface \( R \) such that each component of \( \partial R \) is separating and \( F \cup f(F) \subset \mathring R \).
By our choice of orientation, we can choose \( e_v \in v^{-} \) so that there exists a component \( \partial_0 \) of \( \partial R \) such that \( \partial_0 \) separates \( e_{v} \) from \( \gamma \) and \( \gamma_i \) for each \( i \in \{1, \ldots, n\} \). 

Observe that 
\[
\mathfrak g_R(\gamma) = \left(\sum_{i=1}^n \mathfrak g_R(\gamma_i) \right) - (n-1)\genus(R) - \genus(F)
\]
and, similarly,
\[
\mathfrak g_R(f(\gamma)) = \left(\sum_{i=1}^n \mathfrak g_R(f(\gamma_i)) \right) - (n-1)\genus(R) - \genus(f(F)).
\]
Since \( \genus(F) = \genus(f(F)) \), we see that
\begin{align*}
\vp_v(f) 	&= \mathfrak g_R(f(\gamma)) - \mathfrak g_R(\gamma) \\
		&= \left(\sum_{i=1}^n \mathfrak g_R(f(\gamma_i)) \right) - \left(\sum_{i=1}^n \mathfrak g_R(\gamma_i) \right) \\
		& = \sum_{i=1}^n \left(\mathfrak g_R(f(\gamma_i))-\mathfrak g_R(\gamma_i) \right) \\
		&= \sum_{i=1}^n \vp_{v_i}(f).
\end{align*}
As \( f \in \pmap S \) was arbitrary, the result follows.
\end{proof}

A \emph{principal exhaustion} of \( S \) is a collection of finite-type surfaces \( \{K_i\}_{i\in \bn} \) satisfying:
\begin{enumerate}
\item \( K_i \subset K_j \) whenever \( i < j \),
\item each component of \( \partial K_i \) is separating, and 
\item each component of \( S \ssm K_i \) is of infinite type.
\end{enumerate}

We record a lemma regarding the structure of \( H_1^{sep}(S, \bz) \):

\begin{Lem}
\label{lem:homology1}
If \( \{K_i\}_{i\in\bn} \) is a principal exhaustion of \( S \), then
\[
H_1^{sep}(S, \bz) = \varinjlim H_1^{sep}(K_n, \bz).
\]
In particular, there exists \( n,m \in \bn \) such that every nonzero element \( v \in H_1^{sep}(S, \bz) \) can be written as
\[
v = \sum_{k=1}^m a_k v_k,
\]
where \( a_k \in \bz \) and \( v_k \) can be represented by a peripheral curve on \(  K_n \). 
\end{Lem}

\begin{proof}
The first statement follows directly from the fact that \( S = \varinjlim K_i \).
The second follows from the first and basic facts about the homology of finite-type surfaces. 
\end{proof}

Now fix a principal exhaustion \( \{K_i\}_{i\in\bn} \) of \( S \) and let \( v \in H_1^{sep}(S, \bz) \) be arbitrary.
Write \( v = \sum a_k v_k \) as in Lemma \ref{lem:homology1} and define \( \vp_v \co \pmap S \to \bz \) by 

\begin{equation}
\label{eq:linear}
\vp_v = \sum_{k=1}^m a_k \vp_{v_k}.
\end{equation}

\begin{Lem}
\label{lem:well-defined}
\( \vp_v \) is well defined.
\end{Lem}

\begin{proof}
There are two stages to check: (1) \( \vp_v \) is well-defined with respect to a fixed principal exhaustion and (2) \( \vp_v \) does not depend on the choice of principal exhaustion.

\begin{figure}
\centering
\includegraphics{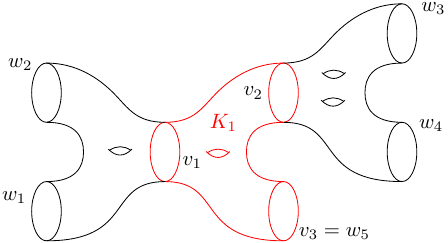}
\caption{An example of \( K_1 \) (red) sitting in \( K_2 \) and the corresponding basis elements.}
\label{fig:bases2}
\end{figure}

(1) Let \( \{K_i\}_{i\in\bn} \) be a principal exhaustion of \( S \) and fix \( n, m \in \bn \) such that \( m > n \). 
Let \( k_n \) and \( k_m \) be the number of boundary components of \( K_n \) and \( K_m \), respectively.
Choose \( v_1, \ldots, v_{k_n-1} \) to be a maximal linearly independent set of simple peripheral elements in \( H_1(K_n, \bz) \) and let \( v_{k_n} = - \sum_{i=1}^{k_n-1} v_i \).
There then exist integers \( n_i \in \{0, \ldots, k_m \} \) for \( i \in \{0,1, \ldots, k_n\} \) with \( n_0 = 0 \) and simple peripheral elements \( w_1, \ldots, w_{k_m} \) of \( H_1(K_m, \bz) \) such that
\[
v_i = \sum_{j = n_{i-1}+1}^{n_i} w_j .
\]
(See Figure \ref{fig:bases2} where \( v_1 \) would be a sum of \( w_1 \) and \( w_2 \).)
In this setup we can apply Lemma \ref{lem:additive} to see that
\begin{equation}
\label{eq:basis}
\vp_{v_i} = \sum_{j = n_{i-1}+1}^{n_i} \vp_{w_j}.
\end{equation}

Now if \( v \in H_1^{sep}(S, \bz) \) such that 
\[
v = \sum_{i=1}^{k_n} a_i v_i  \quad \text{ and } \quad  v =  \sum_{j=1}^{k_m} b_j w_j,
\]
then, by applying \eqref{eq:basis} to each basis element, 
\[
\sum_{i=1}^{k_n} a_i \vp_{v_i} = \sum_{j=1}^{k_m} b_j \vp_{w_j}.
\]
with \( b_j = a_i \) whenever \( n_{i-1} < j \leq n_i \).
This completes (1).

Now let \( \{K_i\}_{i\in \bn} \) and \( \{K_j'\}_{j\in\bn} \) be two principal exhaustions of \( S \).  
Suppose \( v \in H_1^{sep}(S, \bz) \) such that \( v \) can be written in both \( H_1^{sep}(K_n, \bz) \) and \( H_1^{sep}(K_m', \bz) \).
Choose \( N \in \bn \) such that \( K_m', K_n \subset K_N \), then a very similar argument to that of (1) yields the result.
\end{proof}

Neither the integral homology of \( S \) nor the curve graph \( \cc(S) \) can distinguish between an isolated planar end (i.e. puncture) of \( S \) and a boundary component.
Furthermore, the subgroup of \( \pmap S \) generated by Dehn twists about boundary components is contained in \( \cpmap S \).
Therefore, for notational simplicity, we will assume that \( S \) is borderless for the remainder of the section.
Only very slight modifications to the remaining proofs are needed to deal with the case in which \( S \) has compact boundary.
For instance, in the definition of \( \hat S \) below, one would need to additionally cap each boundary component of \( S \) with a disk; this is the main required change.

For an infinite-genus surface \( S \), let \( \hat S \) be the surface obtained by ``forgetting'' the planar ends of \( S \), that is, 
\[
\hat S = \bar S \ssm \E_\infty(S).
\]
Let \( \iota \co S \hookrightarrow \hat S \) denote the inclusion and \( \iota_*\co H_1^{sep}(S, \bz) \to H_1^{sep}(\hat S, \bz) \) the induced homomorphism.

\begin{Prop}[Definition of \( \hat \Phi \)]
\label{prop:definition}
Let \( S \) be an infinite-genus surface.
The map \[ \Phi \co H_1^{sep}(S, \bz) \to H^1(\pmap S, \bz ) \] given by \( \Phi(v) = \vp_v \) is a homomorphism.
Furthermore, \( \Phi \) factors through the homomorphism \( \iota_*\co H_1^{sep}(S, \bz) \to H_1^{sep}(\hat S, \bz) \) yielding a monomorphism \[ \hat \Phi \co H_1^{sep}(\hat S, \bz) \to H^1(\pmap S, \bz). \] 
\end{Prop}

\begin{proof}
The fact that \( \Phi \) is a homomorphism follows immediately from the linearity implicit in the definition of \( \vp_v \) given in \eqref{eq:linear}.
By property (4) of Proposition \ref{prop:homomorphism}, \( \ker \Phi = \ker \iota_* \); hence, \( \hat \Phi \) is a well-defined injective homomorphism.
\end{proof}

The remainder of this section is dedicated to proving that \( \hat \Phi \) is an isomorphism.
To start, we have the following observation whose proof we only sketch as it largely follows from the proof of \cite[Theorem 4]{PatelAlgebraic}.

\begin{Thm}
\label{thm:zero}
Let \( S \) be any surface.
If \( f \in \pmap S \), then \( \vp_v(f) = 0 \) for all \( v \in H_1^{sep}(S, \bz) \) if and only if \( f \in \clpmap S \). 
\end{Thm}

\begin{proof}[Sketch of Proof]
The backwards direction is property (1) from Proposition \ref{prop:homomorphism}. 
For the other implication, suppose \( f \)  has \( \vp_v(f) = 0 \)  for all \( v \in H_1^{sep} (S,\bz) \). 
By exploiting a principal exhaustion of \( S \), the proof of \cite[Proposition 6.2]{PatelAlgebraic} provides a method to write \( f \) as a limit of Dehn twists and handle shifts. 
If this process required a handle shift (that is, if \( \rm{genus}(V ) \neq \rm{genus}(W ) \)
in the last paragraph of the proof of \cite[Proposition 6.2]{PatelAlgebraic}) then we would find \( v \in H_1^{sep} (S,\bz) \)
(namely, in \cite{PatelAlgebraic}, \( v = [a] \) ) with \( \vp_v(f) \neq 0 \). Therefore \( f \in \clpmap S \).
\end{proof}

We can now prove Theorem \ref{thm:decomposition}.

\begin{proof}[Proof of Theorem \ref{thm:decomposition}]
Let \( A_S = \pmap S / \clpmap S \) and let \( H_S = H^1_{sep}( \hat S, \bz ) \). 
Let us start with assuming \( S \) has no planar ends, i.e. \( S = \hat S \).
Let \( r \in \bn \cup \{0, \infty\} \) denote the rank of \( H_1^{sep}(S, \bz) \).
Combining Lemma \ref{lem:homology1} with the fact \( H_1^{sep}( S, \bz) \) is a free abelian group, we see there exists a collection of simple separating homology classes \( \{ v_i \}_{i =1}^r \) such that 
\begin{enumerate}
\item \( H_1^{sep}(  S, \bz ) = \bigoplus_{i=1}^r \langle v_i \rangle \) and
\item there exists a pairwise-disjoint collection of oriented simple closed curves  \( \{\gamma_i\}_{i =1}^r \) on \( S \) such that \( \gamma_i \) represents \( v_i \) for each \( i \in \bn \) and any compact set of \( S \) intersects at most a finite number of the \( \gamma_i \).
\end{enumerate}

It follows that
\[
H_S = \prod_{i=1}^r \langle \pr_i \rangle,
\]
where \( \pr_i\co H_1^{sep}( S, \bz) \to \bz \) is the projection onto the \( i^{th} \) basis element, that is, 

\[
\pr_i(v_j) = \left\{
\begin{array}{ll}
1 & j = i \\
0 & j \neq i
\end{array}
\right.
\]

We will construct a handle shift \( h_i \) that is naturally associated to the projection \( \pr_i \).
Let \( S' \) be the surface obtained from \( S \) by removing pairwise disjoint regular neighborhoods of each \( \gamma_i \).
Observe that each separating curve in \( S' \) bounds a compact surface: if not, then this would contradict \( \{v_i\}_{i=1}^r \) being a basis.
It follows that each component of \( S' \) is one-ended and has infinite-genus.
Every such surface is classified up to homeomorphism by the number of boundary components; for \( n \in \bn \cup \{\infty\} \), let \( Z_n \) be the one-ended infinite-genus surface with \( n \) boundary components.
We can obtain \( Z_n \) by removing \( n \) open discs centered along the horizontal axis in \( \br^2 \) and attaching handles periodically and vertically above each removed disk as shown in Figure \ref{fig:pieces}.

\begin{figure}
\centering
\includegraphics{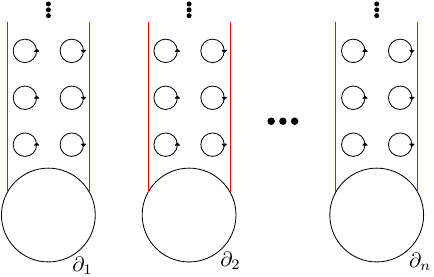}
\caption{The surface \( Z_n \).}
\label{fig:pieces}
\end{figure}

Let \( Y \) be the surface obtained from \( [0,1] \times [0, \infty) \subset \br^2 \) by attaching handles periodically with respect to the map \( Y \hookrightarrow Y \) given by \( (x,y) \mapsto (x, y+1) \).
In our construction of \( Z_n \), there are \( n \) disjoint embeddings of \( Y \) into \( Z_n \) shown in Figure \ref{fig:pieces}.
In each embedding, there is a unique boundary component of \( Z_n \) containing the image of \( [0,1] \times \{0\} \subset Y \).

Fix \( i \in \bn \) and let \( X_1 \) and \( X_2 \) be the two components of \( S' \) containing boundary components \( b_1 \) and \( b_2 \), respectively, homotopic to \( \gamma_i \) in \( S \).
Let \( Y_i \) be the image of \( Y \) in \( X_i \) intersecting \( b_i \).  
As \( b_1 \) and \( b_2 \) bound an annulus in \( S \), we can connect the intervals \( Y_1 \cap b_1 \) and \( Y_2 \cap b_2 \) with a strip \( T \cong [0,1]\times[0,1] \) in the annulus.
Now the connected subsurface \( \Sigma_i = Y_1 \cup T \cup Y_2 \) gives a proper embedding of \( \Sigma \) (see Section \ref{handleshifts}) into \( S \).
Let \( h_i \) be a handle shift supported in \( \Sigma_i \) such that \( \vp_{v_i}(h_i) = 1 \).
Observe that \( \gamma_j \) is disjoint from \( \Sigma_i \) whenever \( i \neq j \); hence, 
\begin{enumerate}[(i)]
\item
\( h_ih_j = h_jh_i \) and
\item
 \( \vp_{v_j}(h_i) = 0 \) whenever \( i \neq j \) (by property (3) of Proposition \ref{prop:homomorphism}).
\end{enumerate}
It follows from (i) that \( \prod_{i=1}^r \langle h_i \rangle < \pmap S \).
We can therefore define the monomorphism \[ \kappa \co H_S \to \pmap S \] by setting \( \kappa(\pr_i) = h_i \). 

Now if \( S \neq \hat S \), then we can repeat the above with \( \hat S \) and by choosing commuting preimages of the \( h_i \) in \( \pmap S \) under the forgetful map \( \pmap S \to \pmap{\hat S} \).
This is always possible as one can choose the \( \Sigma_i \) above to avoid the planar ends of \( S \).

By (1) of Proposition \ref{prop:homomorphism} and Proposition \ref{prop:definition}, every homomorphism \( \hat \Phi(v) \) for \( v \in H_1^{sep}( \hat S, \bz) \) factors through \( A_S \) yielding a homomorphism \( \bar \Phi(v) \co A_S \to \bz \).
Let \( h \in \kappa(H_S) \), so we can write \( h = \prod h_i^{m_i} \).
If \( h \) is not the identity, then there exists \( i \in \bn \) such that \( m_i \neq 0 \).
It follows that \( \hat \Phi(v_i)(h) = m_i \) and thus \( \bar \Phi(v_i)(\pi(h)) = m_i \).
This implies that \( \pi(h) \neq 0 \) and as \( h \) was arbitrary we can conclude that \( \pi \circ \kappa \) is injective.

Given \( a \in A_s \), define \( a^* \in H_S \) by \( a^*(v) = \bar \Phi(v) (a) \).
It is easy to check that the map \( \Psi \co A_S \to H_S \) given by \( \Psi(a) = a^* \) is a homomorphism.
Furthermore, if \( a^* = 0 \), then \( \bar \Phi(v)(a) = 0 \) for all \( v \in H_1^{sep}(\hat S, \bz) \); it follows from Theorem \ref{thm:zero}  that \( a = 0 \) and \( \Psi \) is injective.
Now by the definition of \( h_i \) and property (2) above, we see that \( \pi(h_i)^* = \pr_i \).
The injectivity of \( \pi \circ \kappa \) now implies that the composition \( \pi \circ \kappa \circ \Psi \) is the identity.
We can therefore conclude that \( \pi \circ \kappa \) is an isomorphism  whose inverse is \( \Psi \).
\end{proof}

We can now show that \( \hat \Phi \) is an isomorphism.

\begin{Thm}
\label{thm:isomorphism}
For any surface \( S \) of genus at least two, the map \( \hat \Phi \co H_1^{sep}(\hat S, \bz) \to H^1(\pmap S, \bz) \) is an isomorphism.
\end{Thm}

\begin{proof}
By Proposition \ref{prop:homomorphism}, we already know that \( \hat \Phi \) is injective.
It is left to show surjectivity.
Let \( \vp \in H^1(\pmap S, \bz) \).
If the genus of \( S \) is at least 3, then \( \cpmap S \) is perfect as it is the direct limit of perfect groups (every finite-type pure mapping class group of a surface of genus at least 3 is perfect---see \cite[Theorem 5.2]{FarbPrimer}); hence, \( \vp( \cpmap S ) = 0 \).
If the genus is 2, then \( \clpmap S \) is the direct limit of groups whose abelianization is finite and the same argument applies (see \cite[Section 5.1.2]{FarbPrimer} or \cite[Theorem 5.1]{KorkmazLow}).
By Corollary \ref{cor:polish}, we can apply \cite[Theorem 1]{DudleyContinuity} to see that \( \vp \) is continuous, which allows us to conclude that \( \vp( \clpmap S)  = 0 \).
Therefore, by Corollary \ref{cor:semidirect}, we can identify \( H^1(\pmap S, \bz) \) with \( \Hom(\kappa(H_S), \bz) \).

Using the notation from Theorem \ref{thm:decomposition}, we write
\[
\kappa(H_S) = \prod_{i=1}^r \langle h_i \rangle,
\]
where \( r \) is the rank of \( H_1^{sep}(\hat S, \bz) \) and the \( h_i \) are handle shifts.
Recall from the proof of Theorem \ref{thm:decomposition} that there exists a basis \( \{v_i\}_{i=1}^r \) for \( H_1^{sep}(\hat S, \bz) \) such that
\[
\hat \Phi(v_i)(h_j) = \left\{
\begin{array}{ll}
1 & j = i \\
0 & j \neq i
\end{array}
\right.
\]
Applying standard properties of \( \Hom \) in the finite-rank case or an old result of Specker \cite{SpeckerAdditive} (see also \cite[Lemma 94.1]{FuchsInfinite}) in the case of infinite rank, we see that 
\[
\Hom(\kappa(H_S), \bz) = \bigoplus_{i=1}^r \langle \hat \Phi(v_i) \rangle.
\]
This implies that \( \hat \Phi \) is surjective and hence an isomorphism.
\end{proof}

\bibliographystyle{amsplain}
\bibliography{references}

\providecommand{\bysame}{\leavevmode\hbox to3em{\hrulefill}\thinspace}
\providecommand{\MR}{\relax\ifhmode\unskip\space\fi MR }
\providecommand{\MRhref}[2]{%
  \href{http://www.ams.org/mathscinet-getitem?mr=#1}{#2}
}
\providecommand{\href}[2]{#2}
\begin{thebibliography}{10}

\bibitem{AFP}
Javier Aramayona, Ariadna Fossas, and Hugo Parlier, \emph{Arc and curve graphs
  for infinite-type surfaces}, Proc. Amer. Math. Soc. \textbf{145} (2017),
  no.~11, 4995--5006. \MR{3692012}

\bibitem{AramayonaHomomorphisms}
Javier Aramayona and Juan Souto, \emph{Homomorphisms between mapping class
  groups}, Geom. Topol. \textbf{16} (2012), no.~4, 2285--2341.

\bibitem{AramayonaBig}
Javier Aramayona and Nicholas~G Vlamis, \emph{Big mapping class groups: an
  overview}, arXiv preprint arXiv:2003.07950 (2020).

\bibitem{BavardHyperbolic}
Juliette Bavard, \emph{Hyperbolicit\'e du graphe des rayons et quasi-morphismes
  sur un gros groupe modulaire}, Geom. Topol. \textbf{20} (2016), no.~1,
  491--535. \MR{3470720}

\bibitem{BavardIsomorphisms}
Juliette Bavard, Spencer Dowdall, and Kasra Rafi, \emph{Isomorphisms between
  big mapping class groups}, Int. Math. Res. Not. \textbf{2020} (2020), no.~10,
  3084--3099.

\bibitem{BavardBig}
Juliette Bavard and Anthony Genevois, \emph{Big mapping class groups are not
  acylindrically hyperbolic}, Math. Slovaca \textbf{68} (2018), no.~1, 71--76.
  \MR{3764317}

\bibitem{DomatBig}
George Domat, \emph{Big pure mapping class groups are never perfect}, In
  preparation.

\bibitem{DomatFirst}
George Domat and Paul Plummer, \emph{First cohomology of pure mapping class
  groups of big genus one and zero surfaces}, arXiv preprint arXiv:1904.10565
  (2019).

\bibitem{DudleyContinuity}
R.~M. Dudley, \emph{Continuity of homomorphisms}, Duke Math. J. \textbf{28}
  (1961), 587--594.

\bibitem{DurhamGraphs}
Matthew~Gentry {Durham}, Federica {Fanoni}, and Nicholas~G. {Vlamis},
  \emph{{Graphs of curves on infinite-type surfaces with mapping class group
  actions}}, Ann. Inst. Fourier (Grenoble) \textbf{68} (2018), 2581--2612.

\bibitem{FarbPrimer}
Benson Farb and Dan Margalit, \emph{A primer on mapping class groups},
  Princeton Mathematical Series, vol.~49, Princeton University Press,
  Princeton, NJ, 2012.

\bibitem{FuchsInfinite}
L\'aszl\'o Fuchs, \emph{Infinite abelian groups. {V}ol. {II}}, Academic Press,
  New York-London, 1973, Pure and Applied Mathematics. Vol. 36-II. \MR{0349869}

\bibitem{GasterColoring}
Jonah Gaster, Joshua~Evan Greene, and Nicholas~G. Vlamis, \emph{Coloring curves
  on surfaces}, Forum Math. Sigma \textbf{6} (2018), e17, 42. \MR{3850207}

\bibitem{HernandezAlexander}
Jes{\'u}s~Hern{\'a}ndez Hern{\'a}ndez, Israel Morales, Ferr{\'a}n Valdez,
  et~al., \emph{The alexander method for infinite-type surfaces}, Michigan
  Math. J. \textbf{68} (2019), no.~4, 743--753.

\bibitem{HernandezIsomorphisms}
Jes\'{u}s Hern\'{a}ndez~Hern\'{a}ndez, Israel Morales, and Ferr\'{a}n Valdez,
  \emph{Isomorphisms between curve graphs of infinite-type surfaces are
  geometric}, Rocky Mountain J. Math. \textbf{48} (2018), no.~6, 1887--1904.
  \MR{3879307}

\bibitem{IvanovAutomorphisms}
Nikolai~V. Ivanov, \emph{Automorphism of complexes of curves and of
  {T}eichm\"uller spaces}, Internat. Math. Res. Notices (1997), no.~14,
  651--666. \MR{1460387}

\bibitem{KechrisClassical}
Alexander~S. Kechris, \emph{Classical descriptive set theory}, Graduate Texts
  in Mathematics, vol. 156, Springer-Verlag, New York, 1995.

\bibitem{Kirby}
Rob Kirby, \emph{Problems in low dimensional manifold theory}, Algebraic and
  geometric topology ({P}roc. {S}ympos. {P}ure {M}ath., {S}tanford {U}niv.,
  {S}tanford, {C}alif., 1976), {P}art 2, Proc. Sympos. Pure Math., XXXII, Amer.
  Math. Soc., Providence, R.I., 1978, pp.~273--312. \MR{520548}

\bibitem{KorkmazLow}
Mustafa Korkmaz, \emph{Low-dimensional homology groups of mapping class groups:
  a survey}, Turkish J. Math. \textbf{26} (2002), no.~1, 101--114. \MR{1892804}

\bibitem{LuoAutomorphisms}
Feng Luo, \emph{Automorphisms of the complex of curves}, Topology \textbf{39}
  (2000), no.~2, 283--298. \MR{1722024}

\bibitem{MannAutomatic}
Kathryn Mann, \emph{Automatic continuity for homeomorphism groups and
  applications}, Geom. Topol. \textbf{20} (2016), no.~5, 3033--3056, With an
  appendix by Fr\'ed\'eric Le Roux and Mann. \MR{3556355}

\bibitem{PatelAlgebraic}
Priyam Patel and Nicholas~G. Vlamis, \emph{Algebraic and topological properties
  of big mapping class groups}, Algebr. Geom. Topol. \textbf{18} (2018), no.~7,
  4109--4142. \MR{3892241}

\bibitem{RichardsClassification}
Ian Richards, \emph{On the classification of noncompact surfaces}, Trans. Amer.
  Math. Soc. \textbf{106} (1963), 259--269.

\bibitem{RosendalAutomatic2}
Christian Rosendal, \emph{Automatic continuity in homeomorphism groups of
  compact 2-manifolds}, Israel J. Math. \textbf{166} (2008), 349--367.
  \MR{2430439}

\bibitem{RosendalAutomatic}
Christian Rosendal and S{\l}awomir Solecki, \emph{Automatic continuity of
  homomorphisms and fixed points on metric compacta}, Israel J. Math.
  \textbf{162} (2007), 349--371. \MR{2365867}

\bibitem{SpeckerAdditive}
Ernst Specker, \emph{Additive {G}ruppen von {F}olgen ganzer {Z}ahlen},
  Portugaliae Math. \textbf{9} (1950), 131--140. \MR{0039719}

\end{thebibliography}

\end{document}